 \documentclass{amsart}

\usepackage{mathptmx, amscd, amssymb, amsmath, enumerate, slashbox, hyperref, float}

\usepackage[utf8]{inputenc}
\usepackage[all]{xy}
\usepackage{color}
\usepackage{hyperref}

\usepackage{amsmath, mathtools}
\usepackage{hyperref}
\hypersetup{colorlinks=true}
\usepackage{amssymb}
\usepackage{amsfonts}
\usepackage{graphicx}
\usepackage{mathrsfs}
\usepackage{amscd}
\usepackage[all]{xy}
\usepackage{hyperref}
\usepackage{amsthm}
\usepackage{cleveref}
\usepackage{verbatim}
\usepackage{amscd}
\usepackage{mathtools}
\usepackage{comment}

\newtheorem{thm}{\bf Theorem}[section]

\newtheorem{prop}[thm]{\bf Proposition}
\newtheorem{lemma}[thm]{\bf Lemma}
\newtheorem{cor}[thm]{\bf Corollary}
\theoremstyle{definition}
\newtheorem{definition}[thm]{\bf Definition}
\theoremstyle{remark}
\newtheorem{remark}[thm]{\bf Remark}

\newtheorem{conjecture}[thm]{\bf Conjecture}
\newtheorem{question}[thm]{\bf Question}
\newtheorem{example}[thm]{\bf Example}
\numberwithin{equation}{section}

\newcommand{\HH}[3]{\operatorname{H}^{#1}_{#2}(#3)}

\newcommand\numberthis{\addtocounter{equation}{1}\tag{\theequation}}

\DeclareMathOperator{\height}{{ht}}

\DeclareMathOperator{\depth}{{depth}}

\DeclareMathOperator{\Spec}{{Spec}}

\DeclareMathOperator{\Ass}{{Ass}}
\DeclareMathOperator{\Min}{{Min}}
\DeclareMathOperator{\Tor}{{Tor}}

\DeclareMathOperator{\sell}{{s\ell}}

\DeclareMathOperator{\Cl}{{{Cl}}}

\DeclareMathOperator{\projdim}{{{projdim}}}

\DeclareMathOperator{\ara}{{{ara}}}

\DeclareMathOperator{\supp}{{supp}}

\DeclareMathOperator{\bght}{{bigheight}}
\DeclareMathOperator{\Image}{{Im}}
\DeclareMathOperator{\cx}{{cx}}

\def\ls{\leqslant}
\def\gs{\geqslant}

\def\f0{\mathbf{0}}

\def\fx{\mathbf{x}}

\def\fp{\mathbf{p}}

\def\fa{\mathbf{a}}

\def\fm{\mathfrak{m}}

\def\fp{\mathfrak{p}}

\def \CC{\mathbb C}
\def \RR{\mathbb R}

\def \ZZ{\mathbb Z}
\def \NN{\mathbb N}

\def \C{\mathcal C}

\def \G{\mathcal G}
\def \R{\mathcal R}

\def \I{\mathcal I}

\newcommand{\ses}[3]{0 \to {#1} \to {#2} \to {#3} \to 0}

\begin{document}

\title[Symbolic Analytic Spread:  Upper Bounds and Applications]{Symbolic Analytic Spread: Upper Bounds and Applications}

\author[Hailong Dao]{Hailong Dao}
\address{Hailong Dao\\ Department of Mathematics \\ University of Kansas\\405 Snow Hall, 1460 Jayhawk Blvd.\\ Lawrence, KS 66045}
\email{hdao@ku.edu}

\author[Jonathan Monta\~no]{Jonathan Monta\~no}
\address{Jonathan Monta\~no \\ Department of Mathematical Sciences  \\ New Mexico State University  \\PO Box 30001\\Las Cruces, NM 88003-8001}
\email{jmon@nmsu.edu}

\begin{abstract}
The symbolic analytic spread of an ideal $I$ is defined in terms of the rate of growth of the minimal number of generators of its symbolic powers. In this article we find upper bounds for the symbolic analytic spread under certain conditions in terms of other invariants of  $I$. Our methods also work for more general systems of ideals. As applications we provide bounds for the (local) Kodaira dimension of divisors, the arithmetic rank, and the Frobenius complexity. We also show sufficient conditions for an ideal to be a set-theoretic complete intersection. 
\end{abstract}


\maketitle

\section{Introduction}
The symbolic powers of an ideal $I$ in a Noetherian commutative ring $R$ have long been studied in commutative algebra and algebraic geometry, see \cite{DaoSurvey} for a survey. In recent years, this subject has enjoyed an intense resurgence, thanks to exciting new connections and methods discovered by a number of researchers (see for instance \cite{DGJ}, \cite{GH}, \cite{MS}, \cite{MNB},  \cite{NT}, \cite{PST}).  

Despite these developments, many basic questions about symbolic powers remain tantalizingly open. For example, we do not know if the number of generators of the $n$-th symbolic power $I^{(n)}$ of a homogenous ideal 
$I$ in a  polynomial ring is  bounded above by a polynomial in $n$. The purpose of this article is to start a systematic study of the rate of growth  of $\mu(I^{(n)})$, as $n$ varies, for general ideals in a local ring.  Such invariant has been considered in specific contexts under the name {\it symbolic analytic spread} of $I$. We adopt this terminology and denote this invariant by $\sell(I)$ (see Definition \ref{keydef}).

In this article, we are able to show tight upper bounds for the symbolic analytic spread using the usual analytic spread $\ell(I)$, or the dimension of $R/I$, under certain conditions (for instance if $\dim R/I\ls 3$, if $\depth R/I^{(n)}\gs \dim R/I-1$ eventually, or if $I$ is locally a complete intersection on the punctured spectrum). While these conditions are restrictive, as far as we know the results are new even for prime ideals over a regular local ring. Previously, in \cite{Dutta}, Dutta had found a similar upper bound for $\sell(I)$ under the assumption that $R$ is $S_2$, $\dim R/I\ls 2$, and $I$ is the intersection of prime ideals of the same height. 
It is also worth noting that  our methods work for general systems of ideals containing the powers of $I$,  and with a view toward further applications we state our main results in this framework. See Theorems \ref{gen} and \ref{duttaLike} for the key general statements; and Corollaries \ref{depthCor}, \ref{ciCor}, \ref{duttaCor}, and Proposition  \ref{extr}, for particular results about symbolic powers. 

We proceed to apply our bounds in a number of situations. For $I$ locally principal on the punctured spectrum (i.e., $I$ represents a line bundle on $\Spec R \setminus \{\fm\}$), the symbolic  analytic spread can be viewed as a local version of the Kodaira dimension associated to this divisor, and our bounds imply that this is less than or equal to $\ell(I)$ (see Corollary \ref{divisor}). We also give bounds on the Frobenius complexity of $R$ (see Corollary \ref{FrobCor}). Finally, we give a criterion for an ideal to be set-theoretic complete intersection, extending classical results on this problem (see Corollary \ref{SCICor}). We discuss the case of monomial ideals in Section \ref{monoSec} and prove a sharp bound for $\sell(I)$ in this case, which can be used to recover a bound by Lyubeznik on the arithmetic rank of monomial ideals, and to provide a lower bound for the depth of symbolic powers of squarefree monomial ideals (see Theorem \ref{mainMonomial}, and Corollaries \ref{LyuCor} and \ref{depthBound}).

We end the paper with a number of open questions and conjectures, whose main purpose is to highlight how much we still do not know about the symbolic analytic spread.

\section{Notation and preliminary results}\label{FirstS}

Let $(R,\fm,k)$ be a Noetherian local ring and  $\I=\{I_n\}_{n\in \NN}$  a system of ideals of $R$, with $I_0=R$. We often assume that the ideals of $\I$ contain the powers of an ideal $I$ of $R$, i.e., $I^n\subseteq I_n$ for every $n\in \NN$. A natural example where this condition is satisfied is when $\I$ consists of symbolic powers.

Let $I$ and $J$ be ideals of $R$, for any $n\in \NN$ we call the ideal $I^{(n)}_J:=(I^n:_RJ^\infty)$ the {\it $n$th-symbolic power of $I$ with respect to $J$}. Let $\Ass^\infty(I):=\lim_{n\to \infty} \Ass(I^n)$, which exists and is a finite set by \cite{Br2}. 
We note that if $J$ is the intersection of the  prime ideals in $ \Ass^\infty(I)\setminus \Min(I)$ then for $n\gg 0$, the ideal $I^{(n)}_J$ is equal to $I^{(n)}$ the {\it $n$th-$($ordinary$)$ symbolic powers of $I$}, i.e., $$I^{(n)}=\displaystyle\bigcap_{\fp \in \Min(I)} (I^nR_\fp \cap R);$$
 whereas if $J=\fm$, then $I^{(n)}_J$ is equal to $\{(I^n)^{sat}\}_{n\in \NN}$, the  saturated powers of $I$. 
 
For two sequences of non-negative real numbers  $\{a_n\}_{n\in \NN}$ and  $\{b_n\}_{n\in \NN}$, we write $a_n = O(b_n)$ if $\limsup_{n\rightarrow \infty}\frac{a_n}{b_n}<\infty$. For a finitely generated $R$-module $M$, we denote by $\mu(M)$ its  minimal number of generators.

\begin{definition}[cf. {\cite{BSch}, \cite{HKTT}}]\label{keydef}
We   define the {\it analytic spread} of the system $\I$, and denote it by $\ell(\I)$, as
$$\ell(\I)=\min\{t\in \RR\mid \mu(I_n) = O(n^{t-1})\}.$$
 %If not such $t$ exists, we set $\ell(\I)=\infty$.  
If $\I=\{I^{(n)}\}_{n\in \NN}$,  we denote $\ell(\I)$  by $\sell(I)$ and call it the {\it symbolic analytic spread of $I$}. We also note that if  $\I=\{I^n\}_{n\in \NN}$, then $\ell(\I)$ is simply $\ell(I)$, the {\it analytic spread} of $I$.
\end{definition}

% whereas if $J=\fm$, they are equal to $\{\tilde{I^n}\}_{n\in \NN}$, the  saturated powers of $I$. 

%The {\it $n$th-symbolic power} of $I$ is defined as $$I^{(n)}:=\displaystyle\bigcap_{\fp \in \Min(R/I)} (I^nR_\fp \cap R).$$ 

Let $\R^s(I):=\oplus_{n\in \NN} I^{(n)}t^n $ dentote the {\it symbolic Rees algebra of $I$}. We note that $\R^s(I)$ is Noetherian whenever $I$ is a  monomial ideal \cite[3.2]{HHT} or a determinantal  ideal \cite[10.2, 10.4]{BV}. However, this is not always the case even for prime ideals in regular rings \cite{Ro}.

When $\R^s(I)$ is Noetherian we can extract from it an upper bound for $\sell(I)$, which in turn provides a bound for the {\it arithmetic rank} of $I$, i.e., $$\ara(I)=\min \{u\mid \text{there exist }g_1,\ldots, g_u\in R \text{ such that } \sqrt{(\underline{g})}=\sqrt{I}\}.$$
We summarize these facts in the following proposition, but first we recall the following  fact.

\begin{remark}
For any $R$-ideal $I$ we have $\ara(I)\ls \ell(I)$ (see for example \cite[8.3.8]{HS}).
\end{remark}

\begin{prop}\label{ariRank}
Let $d=\dim R$. Assume $\depth R\gs \min\{\dim R/I, 1\}$ and that    $\R^s(I)$ is Noetherian. Then $$\ara(I)\ls \sell(I)\ls \max\{d-  \dim R/I,\, d-1\}.$$
\end{prop}

\begin{proof} 
Set $T=k\otimes_R \R^s(I)$. Since $T$  is Noetherian, there exists $c\in \ZZ_{>0}$ such that the $c$th-Veronese subalgebra $T(c)=\oplus_{n\in \NN}T_{cn}\subseteq T$ is a standard graded $R$ algebra and $T$ is a finitely generated  $T(c)$-module (\cite[2.1]{HHT}, \cite{Ratliff}). Set $J=I^{(c)}$, then $\sell(I)= \dim T=\dim T(c)=\ell(J)\ls d$. Since $\sqrt{J}=\sqrt{I}$, we  have $\ara(I)=\ara(J)\ls \ell(J) =\sell(I)$. Now, since $J^{(n)}=J^n$ for every $n\in \NN$, we have  $\Ass(J^n/J^{n+1})\subseteq \Ass(R/J^{n+1})=\Min(R/J)$. This implies that the algebra $\G(J):=\oplus_{n\in \NN}J^n/J^{n+1}$ is torsion-free over $R/J$. If $\dim R/J = \dim R/I\gs 1$ and $\depth R\gs 1$, it follows that $\fm$ is not contained in any minimal prime of  $\G(J)$.  Hence, $\ell(J)=\dim k\otimes_R \G(J)<\dim  \G(J)=d$, finishing the proof.
\end{proof}

\begin{remark}\label{someC}
 From the proof of Proposition \ref{ariRank} it can be seen that if  $T = \R^s(I)\otimes_R k=\oplus_{n\in \NN}I^{(n)}/\fm I^{(n)}$ is Noetherian, then $\sell(I)=\dim T$. Moreover, there exists $c\in \ZZ_{>0}$ such that $\sell(I)=\ell(I^{(nc)})$ for every $n\in \ZZ_{>0}$; in fact, this holds for every $c\gg 0$ \cite[2.4.4]{Ratliff}.
\end{remark}

The following example shows that the  bound for the arithmetic rank in terms of the symbolic analytic spread  sometimes is sharper than the one provided by the analytic spread. 

\begin{example}(\cite{BSch})\label{determ}
Let $X$ be a generic matrix of size $m\times n$. Let $R=k[X]$ be the polynomial ring in the entries of $X$ over the field $k$ and $\fm$ its irrelevant maximal ideal. We denote by $I_t(X)$ the {\it determinantal} ideal generated by the $t$-minors of $X$. Assume $k$ is algebraically closed, then for every  $t\ls m$ we have
$$\sell(I_t(X))=\ara I_t(X) = mn-t^2+1.$$
However, if $t<\min\{m,n\}$ then $\ell(I_t(X))=mn$. 
\end{example}

For an $R$-module $M$, we denote by $e(M)=\lim_{n\rightarrow \infty}\frac{(\dim M)!\lambda(M/\fm^nM)}{n^{\dim M}}$ the {\it multiplicity} of $M$. We recall the following  well-known fact.

\begin{remark}\label{CMbound}
If $M$ is a Cohen-Macaulay $R$-module, then $\mu(M)\ls e(M)$.
\end{remark}

The next example shows that the Noetherian assumption on $\R^s(I)$ is necessary in Proposition \ref{ariRank}. 

\begin{example}\label{notariBound}
 Let $$R=\CC[[x,y,z]]/(x^3+y^3+z^3).$$ Since $R$ is a normal two-dimensional domain that is not rational singularity, the divisor class group $\Cl(R)$ is infinite \cite[17.4]{Lipman69}. Let $\fp\in\Spec R$ be such that  $[\fp]\in \Cl(R)$ has infinite order. Since $\fp^{(n)}$ is a Maximal Cohen-Macaulay module of rank one for every $n\gs 1$, by \cite[4.7.9]{BH} we must have $e(\fp^{(n)})= e(R)=3$. Then,  $\mu(\fp^{(n)})\ls 3$ by Remark \ref{CMbound}. Therefore, $\sell(\fp)=1$, but by assumption $\fp$  cannot be the radical of a principal ideal and then $\ara(\fp)>1$.
\end{example}

\section{Main results}

This section includes the main results of this paper. We follow the notation introduced in Section \ref{FirstS}. In particular,  we assume $(R,\fm,k)$ is a Noetherian local ring and $I$ an $R$-ideal.

We begin with the following result that is of independent interest. Here, we extend the bound in Remark \ref{CMbound} for non-Cohen-Macaulay modules.  We recall that  $x\in \fm$ is  a {\it superficial element} of $\fm$ with respect to the $R$-module $M$ if there exists $c\in \NN$ such that for every integer $n\gs c$ we have $(\fm^{n+1}M:_M x)\cap \fm^c M=\fm^nM$. The sequence $x_1,\ldots, x_s$ is {\it superficial}  of $\fm$ with respect to $M$ if for every $i=1,\ldots, s$, the image of $x_i$ in $\overline{R}:=R/(x_1,\ldots, x_{i-1})$ is a superficial element of $\fm\overline{R}$ with respect to $M\otimes_R \overline{R}$.

%Let $(R,\fm, k)$ be a analytically unramified and formally equidimensional  local ring of dimension $d$ and $I$ an $R$-ideal.

\begin{prop}\label{keybound}
Let $r=\dim M\gs 1$. Then for every superficial sequence $\underline x= x_1,\dots, x_{r-1}$ of $\fm$ with respect to $M$ we have $$\mu(M) \ls \mu(\HH{0}{\fm}{M/(\underline x)M}) + e(M) $$
\end{prop}

\begin{proof}
Set $\overline{M}=M/(\underline x)M$ and note $\dim \overline{M}=1$ \cite[Proposition  1.2(1)]{RV}. Then,
$$\mu(M)=\mu(\overline{M})\ls \mu(\HH{0}{\fm}{\overline{M}})+ \mu(\overline{M}/\HH{0}{\fm}{\overline{M}})\ls \mu(\HH{0}{\fm}{\overline{M}})+ e(\overline{M}/\HH{0}{\fm}{\overline{M}}),$$ 
where the first inequality follows from the short exact sequence $$\ses{\HH{0}{\fm}{\overline{M}}}{\overline{M}}{\overline{M}/\HH{0}{\fm}{\overline{M}}}$$ and the second one from Remark \ref{CMbound}. The conclusion now follows since $e(\overline{M}/\HH{0}{\fm}{\overline{M}})=e(\overline{M})$ by \cite[4.7.7]{BH} and $e(\overline{M})=e(M)$ by  \cite[Proposition 1.2(2)]{RV}.
\end{proof}

\begin{example}
 In the following cases the equality in Lemma \ref{keybound} holds.
 \begin{enumerate}
 \item[(1)] $M$ is an {\it Ulrich} $R$-module, i.e., $M$ is Cohen-Macaulay and $\mu(M)=e(M)$.
 \item[(2)]  $R$ is regular, $M$ is torsion-free, and $\underline{x}$ is part of a regular system of parameters. Indeed, in this case the ring $\overline{R}:=R/(\underline{x})$ is a DVR and then $\overline{M}:=M/(\underline x)M\cong \overline{R}^s\oplus N$ for some $s\gs 1$ and $N$ a zero-dimensional  $\overline{R}$-module.  Therefore, $\mu(M)=\mu(\overline{M})=s+\mu(N)$. On the other hand, $e(M)=e(\overline{M})=e(\overline{R}^s)=s$ and $\mu(\HH{0}{\fm}{\overline{M}}) =\mu(N)$.
 
 \end{enumerate}
\end{example}

For clarity of exposition, we define the following object depending of the the value of $\ell(I)$. 

\begin{definition}
$\tilde{\ell}(I)=
\begin{cases}
\ell(I)&\text{ if } \ell(I)<\dim(R)\\
d+1& \text{ if } \ell(I)=\dim(R)
\end{cases}.$
\end{definition}

We need one more lemma prior stating the first main result.

\begin{lemma}\label{sats}
Assume $R$ is  analytically unramified and formally equidimensional. Then, $\lambda(\HH{0}{\fm}{R/I^n})=O(n^{\tilde{\ell}(I)-1})$.
\end{lemma}
\begin{proof}
Set $d=\dim(R)$. If $\ell(I)=d$, the statement follows by \cite[4.7]{KV}. We then assume $\ell(I)<d$ and thus, by McAdam's Theorem \cite[5.4.6]{HS}, we have $(\overline{I^n})^{sat}=\overline{I^n}$, where $\overline{J}$ denotes the integral closure of the ideal $J$. Therefore, $(I^n)^{sat}\subseteq \overline{I^n}$ for all $n> 0$. Set $\R=\oplus_{n\in \NN} I^nt^n$, $ \R^{sat}= \oplus_{n\in \NN} (I^n)^{sat}t^n$, and  $\overline{\R}=\oplus_{n\in \NN}\overline{I^n}t^n$. Therefore, $\R \subseteq \R^{sat} \subseteq \overline{\R}$, and then $\R^{sat}$ is a finitely generated $\R$-module \cite[9.2.1]{HS}. It follows that there exists an $N$ such that $\fm^N \R^{sat} \subseteq \R$. Thus, $\R^{sat}/ \R$ is a a finite $\R/\fm^N \R$-module. Then $\lambda(\HH{0}{\fm}{R/I^n})=\lambda((I^n)^{sat}/I^n)= \lambda\big([\R^{sat} / \R]_n\big)$ agrees with a polynomial of degree at most $\dim \R/\fm^N\R-1 = \ell(I)-1$ for $n\gg 0$. 
\end{proof}

We are now ready to present our first main theorem.

\begin{thm}\label{gen}
Assume $R$ is  analytically unramified and formally equidimensional. Let   $\I=\{I_n\}_{n\in \NN}$ be a system of ideals such that
\begin{enumerate}
\item $I^n\subseteq I_n$ for all $n\gs 0$, and
\item $\depth R/I_n\gs \dim (I_n/I^n)-1$ for $n\gg 0$. 
\end{enumerate}

Then,
%\begin{enumerate}
$\ell(\I)\ls \ell(I)+1$. 
%\item  $\reg(I_n) = O(n)$. 
%\end{enumerate}

\end{thm}

\begin{proof}
%Tensoring by the $\fm$-adic completion $\hat{R}$ of $R$  we can assume $R$ is complete. Moreover, by possibly extending the residue field we can assume  $k$ is infinite. 

Set $M_n = I_n/I^n$.  
 %We first prove (1). 
 From $\ses{I^n}{I_n}{M_n}$ we obtain $\mu(I_n) \ls \mu(I^n)+\mu(M_n)$, so it suffices to show $\mu(M_n) = O(n^{\ell(I)})$ which we prove  using  Proposition \ref{keybound}. In order to construct a sequence that is  superficial  simultaneously for all the $M_n$, we need to introduce a faithfully flat extension of $R$.  
 
 Let $f_1,\ldots, f_u$ be a minimal set of generators of $\fm$ and  $r = \dim R/I$. Consider a set of $ru$ variables $\underline{z}=\{z_{i,j}\}_{1\ls i\ls r,\, 1\ls j\ls u}$ and $R' = R[\underline{z}]_{\fm R[\underline{z}]}$. We also consider the elements  $x_i = \sum_{j=1}^u z_{i,j}f_j\in R'$ for $i=1,\ldots, r$. For every $R$-module $M$ we denote by $M'$ the $R'$-module $M\otimes_R R'$ and we note that $\mu(M)=\mu( M')$. It suffices to show $\mu(M_n')=O(n^{\ell(I)})$.
 
 For each $i=0,\ldots, r$ we consider the set $\Lambda_i=\{n\mid \dim M_n'=i\}$. We  show that for each $i$, we have $\mu(M_n')=O(n^{\ell(I)})$ for $n\in \Lambda_i$. By Lemma \ref{sats}, for $n\in \Lambda_0$ we have $\mu(M_n')\ls \lambda(M_n')\ls \HH{0}{\fm}{R'/(I')^n}=O(n^{\ell(I)})$. Now, fix $i\gs 1$ and consider the modules $M'_n$ for $n\in \Lambda_i$. We may also assume $\Lambda_i$ is an infinite set.  From the short exact sequence 
 \begin{equation}\label{theSes}
 \ses{M_n'}{R'/(I')^n}{R'/I_n'}
\end{equation} 
it follows that $ \Min(M_n')\subseteq \Ass^\infty(I')$. In particular, $\cup_{n\in \Lambda_i}\Min(M_n')$ is a finite set. By the associativity formula \cite[4.7.8]{BH} we have 
$$e(M_n')=\sum_{\fp}\lambda((M'_n)_\fp)e(R'/\fp)$$ where the sum  ranges   over $\fp\in \cup_{n\in \Lambda_i}\Min(M_n')$ with $\dim R'/\fp=i$. For each of these $\fp$ and $n$ we have $\dim (M'_n)_\fp\ls 0$, therefore 
\begin{equation}\label{geq1}
\lambda((M'_n)_\fp)\ls \lambda \big(\HH{0}{\fp R'_{\fp}}{R'_{\fp}/(I'_{\fp})^n}\big)  =O(n^{\ell(I')})=O(n^{\ell(I)}),
\end{equation} where the first equality holds by Lemma \ref{sats} and the fact that the analytic spread does not increase after localization.  

For every $n\in \Lambda_i$, we have that $\underline{x}=x_1,...,x_{i-1}$ is a superficial sequence of $\fm R'$ with respect to $M'_n$ (see for instance \cite[Theorem 1.2]{RV}). 
Moreover, by the assumption (2), $\underline{x}$ is regular on $R'/I_n'$ for $n\gg 0$ and then  $\Tor_1^{R'}(R'/I_n',R'/(\underline{x}))=0$. 
Thus,  tensoring \eqref{theSes} with $\overline{R'}:=R'/(\underline{x})$ we conclude that $M_n'/(\underline{x})M_n'$ embeds into  $ \overline{R'}/I^n\overline{R'}$. Hence,
\begin{equation}\label{geq2}
\mu\big(\HH{0}{\fm}{M_n'/(\underline{x})M_n'}\big) \ls \lambda\big(\HH{0}{\fm}{M_n'/(\underline{x})M_n'}\big)\ls \lambda\big(\HH{0}{\fm}{\overline{R'}/I^n\overline{R'}}\big)=O(n^{\ell(I)}),
\end{equation}
again by Lemma \ref{sats}. The conclusion now follows from \eqref{geq1}, \eqref{geq2}, and Proposition \ref{keybound}.
\end{proof}

From   Theorem  \ref{gen} we obtain the following bounds for the symbolic analytic spread.

\begin{cor}\label{depthCor}
Let $R$ and $I$ be as in Theorem \ref{gen}. If $\depth R/I^{(n)}\gs \dim R/I-2$ for every $n\gg 0$, then $\sell(I)\ls \ell(I)+1$. In particular, this holds if $\dim R/I\ls 3$.
\end{cor}
\begin{proof}
The result follows from Theorem \ref{gen} since $\dim I^{(n)}/I< \dim R/I$.
\end{proof}

\begin{cor}\label{ciCor}
Let $R$ and $I$ be as in Theorem \ref{gen}. If $J$ is a proper ideal such that $\dim R/J\ls 2$ and $\I=\{I^{(n)}_J\}_{n\in \NN}$, then $\ell(\I)\ls \ell(I)+1$. In particular,  if $I$ is a complete intersection locally in codimension $\dim R -3$, then $\sell(I)\ls \ell(I)+1$.
\end{cor}

\begin{proof}
We may assume $\dim R/I\gs 1$.  Then we have $\dim I^{(n)}_J/I^n \ls \dim R/J\ls 2$ for every $n\gs 1$ and $\depth R/I^{(n)}_J \gs  1$. The result now follows by Theorem \ref{gen}. 

For the second statement, we note that this condition implies that for every  $\fp \in \Ass^\infty(I)\setminus \Min(I)$ we have $\dim R/\fp \ls 2$.
\end{proof}

Under some extra conditions we are able to provide a  better bound for the symbolic analytic spread.

\begin{prop}\label{extr}
Assume $R$ is  analytically unramified and formally equidimensional, and that any of the following  conditions holds. 
\begin{enumerate}
\item[(1)] $\ell(I_\fp)<\height \fp$  for every $\fp \in V(I)\setminus\Min(I)$  with $\fp \neq \fm$ and $I^n_\fp$ is integrally closed for every $\fp\in \Min(I)$ and $n\gg 0$.
\item[(2)]  $R$ is a domain and $\ell(I_\fp)<\height \fp$  for every $\fp\in \Spec R$ that contains a prime ideal in $\Ass^\infty(I)\setminus \Min(I)$.
\end{enumerate}
Then $\sell(I)\ls \ell(I)$.
\end{prop}
\begin{proof}
We begin with (1). Let $\overline{J}$ denote the integral closure of the ideal $J$. The assumptions on $I$ imply $\overline{I^n}\subseteq I^{(n)}$  and $\lambda(I^{(n)}/\overline{I^n})<\infty$ for every  $n\gg 0$ \cite[2.4, 2.8]{Gi}.  Since $\overline{I^n}$ is {\it $\fm$-full}, we have $\mu(I^{(n)})\ls \mu(\overline{I^n})$ \cite[2.2]{Goto}. The conclusion now follows by noticing that $\mu(\overline{I^n})=O(n^{\ell(I)-1})$ \cite[9.2.1]{HS}.

We continue with (2). Under these assumptions in \cite[2.7]{AP} it is proved, via \cite[2.1]{CHS}, that $\R^s(I) $ is Noetherian and $\ell(I^{(n)})\ls \ell(I)$ for every $n\gg 0$. The conclusion now follows by Remark \ref{someC}.
\end{proof}

\begin{remark}\label{rem_extr}
The  assumptions  in Proposition \ref{extr} (1) are satisfied if $R$ is regular, $R/I$ is reduced, and $\Ass^\infty(I)\subseteq \Min(I)\cup\{\fm\}$. The latter condition holds, in particular, if $R/I$ is locally a complete intersection.
\end{remark}

It is worth noting a very particular case of Proposition \ref{extr} (1). Namely, when $I$ is locally a  principal ideal on $\Spec R\setminus \{\fm\}$. In other words, $I$ represents a Cartier divisor on the punctured spectrum. In this case, the symbolic analytic spread can be viewed as the local version of the Kodaira dimension of that divisor. In view of this we state the next corollary, which  we also apply in  Subsection \ref{FrobC}. 

\begin{cor}\label{divisor}
Assume $R$ is  analytically unramified and formally equidimensional. Assume that  $I_{\fp}$ is a  principal ideal for  every $\fp \in \Spec R\setminus \{\fm\}$ and that $R_\fp$ is normal for every $\fp\in \Min(I)$. Then $\sell(I)\ls \ell(I)$. 
\end{cor}

The following is our second main theorem. In this result, we obtain a better bound for the analytic spread of a system of ideals under stronger assumptions on the depth of its ideals. This resullt is inspired by Dutta's original proof in \cite{Dutta}. 

\begin{thm}\label{duttaLike}
%Let $(R, \fm)$ be a local ring. 
Let $\I=\{I_n\}_{n\in \NN}$ be a system of ideals such that $ I^n\subseteq I_n$ for all $n\gs 0$. Set $r=\dim R/I$.
\begin{enumerate}
\item If $\depth R\gs r$ and $\depth R/I_n\gs r-1$ for $n\gg 0$, then $\ell(\I) \ls \dim R-r+1$.
\item   If $\depth R\gs r+1$ and  $\depth R/I_n = r$, i.e., $R/I_n$ is Cohen-Macaulay, for $n\gg 0$, then $\ell(\I) \ls \dim R-r$.
\end{enumerate}

\end{thm}

\begin{proof}
We begin with (1). After passing to  the $\fm$-adic completion $\hat{R}$ of $R$  all of the assumptions are preserved, then we can assume $R$ is complete. By countable prime avoidance \cite[Lemma 3]{Burch}, there exists an $R$-regular sequence $\underline x= x_1,...,x_{r}$ such that $R/(I,\underline x)$ has finite length, and $x_1,...,x_{r-1}$ is regular on $R/I_n$ for $n\gg 0$. %Under assumption (2) we can further assume $\underline{x}$ is regular on $R/I_n$ for $n\gg 0$.  

Tensoring the short exact sequence $ \ses{\fm/{\underline x}}{R/{\underline x}}{R/\fm}$
with $R/I_n$ we obtain 
\begin{align*}
\mu(I_n) =\Tor_1^R(R/I_n,R/\fm) &\ls \lambda(R/I_n \otimes_R \fm/{\underline x}) + \lambda(\Tor_1^R(R/I_n, R/\underline x))\\
&\ls\lambda(R/I^n \otimes_R \fm/{\underline x}) + \lambda(\Tor_1^R(R/I_n, R/\underline x)). \numberthis \label{theeqn}
\end{align*} We now show that the  two terms of \eqref{theeqn} are $O(n^{d-r})$. For the first term this follows from \cite[4.6.2]{BH} as %we note that  $\lambda (R/I^n\otimes_R \fm/{\underline x})= O(n^{d-r})$ since 
$\dim  \fm/{\underline x} = d-r$. For the second term,  by \cite[1.6.17]{BH}  we have $\Tor_i^R(R/I_n, R/\underline x)=0$ for  every $i>1$ and $n\gg 0$. Thus, by the non-negativity of Euler characteristic \cite[4.7.4]{BH}, we have $$\lambda(\Tor_1^R(R/I_n, R/\underline x)) \ls \lambda(R/I_n \otimes_R R/{\underline x}) \ls \lambda(R/I^n \otimes_R R/{\underline x}),$$ and this last term is $O(n^{d-r})$ by \cite[4.6.2]{BH}. This concludes the proof of (1).

The proof of (2) follows the  same steps as (1), but considering a longer $R$-regular sequence $x_1,\ldots, x_{r+1}$.
\end{proof}

\begin{cor}\label{duttaCor}
%Let $(R, \fm)$ be a local ring and $I$ be an ideal.
 Set $r=\dim R/I$. 
\begin{enumerate}
\item[(1)] If $\depth R\gs r$ and $\depth R/I^{(n)} \gs r-1$ for $n\gg 0$, then $\sell(I) \ls \dim R-r+1$.
\item[(2)]   If $\depth R\gs r+1$ and  $R/I^{(n)}$ is Cohen-Macaulay for $n\gg 0$, then $\sell(I) \ls \dim R-r$.
\end{enumerate}
\end{cor}

In the following subsections we give two important applications of our results.

\subsection{Set-theoretic complete intersection}

Let $I$ be an $R$-ideal and $g_1,\ldots, g_u$ elements of $R$ such that $\sqrt{I}=\sqrt{\underline{g}}$. By Krull's Altitude Theorem it follows that $ \height I\ls u$. Therefore we always have an inequality $\height I\ls \ara(I)$. The ideal $I$ is said to be a {\it set-theoretic complete intersection} if the equality occurs.

A well-known result of Cowsik, which follows from Proposition \ref{ariRank}, states that if $R$ is a regular local ring and $I$ is a one-dimensional ideal such that $\R^s(I)=\oplus_{n\in \NN}I^{(n)}t^n$ is Noetherian, then $I$ is a set-theoretic complete intersection.   In a related result \cite[2.9]{Varb}, Varbaro shows that if $R=k[x_1,\ldots, x_d]_{(x_1,\ldots, x_d)}$ is the localization of a polynomial ring and $I$ is the {\it Stanley-Reisner ideal} of a {\it matroid}, then $I$ is a set-theoretic complete intersection. 

The following corollary  extend these results.

\begin{cor}\label{SCICor}
Assume $R$ is formally equidimensional and $\depth R\gs \dim R/I +1$ (e.g. $R$ is Cohen-Macaulay and $\height I\gs 1$). If $\R^s(I)$ is Noetherian and $R/I^{(n)}$ is Cohen-Macaulay for $n\gg 0$, then $I$ is a set-theoretic complete intersection.
\end{cor}
\begin{proof}
By Proposition \ref{ariRank} and  Corollary \ref{duttaCor}  we have $$\ara(I)\ls \sell(I)\ls \dim R - \dim R/I=\height I\ls \ara(I),$$
the result follows.
\end{proof}

\subsection{The Frobenius complexity}\label{FrobC}

In this subsection we assume $(R,\fm,k)$ is a Noetherian local ring of characteristic $p>0$.  For every $e\in \NN$ we consider $\C_e^R$, the set of additive maps $\varphi:R\to R$ such that $\varphi(r^{p^e}m)=r\varphi(m)$ for every $r,m\in R$. The direct sum $\C^R=\oplus_{n\in \NN}\C_e^R$ has a graded algebra structure and is called the {\it total Cartier algebra on $R$} \cite[6.1]{KSSZ}. 

For every $e\in \ZZ_{>0}$, let $G_e$ be the subalgebra of $\C^R$ generated by the elements of degree $\ls e$ and  $c_e(R)$  the minimal number of generators of the $R$-module of $\C^R_e/[G_{e-1}]_e$. The  {\it Frobenius complexity} of $R$ is defined as 
$$\cx_F(R)=\inf\{\alpha\in \RR \mid c_e(R) = O(p^{\alpha e}) \}.$$
We remark that here we are using the alternative definition of Frobenius complexity as in \cite[Section 4]{EY2}, which was later adopted  in \cite{Page} and \cite{EP}. If $R$ is $F$-finite and complete this definition coincides with the original one introduced by Enescu and Yao in \cite{EY1}.

In \cite[4.11]{EY1} the following question is raised.

\begin{question}[Enescu-Yao]\label{EYquestion}
Assume $R$ is complete and normal, is $\cx_F(R)$ finite?
\end{question}

In \cite{EP}, the authors show that the Frobenius complexity is finite for standard graded rings over an $F$-finite field localized at the irrelevant maximal ideal. However, in general $\cx_F(R)$ is not known to be finite, except when the {\it anticanonical cover}  is Noetherian \cite[4.7]{EY1}, when $\dim R\ls 2$ (in this case $\cx_F(R)\ls 0$) \cite[4.10]{EY1}, and in other particular cases \cite{EY2,Page}.

The following proposition, combined with our main results, allows us to answer Question \ref{EYquestion} for rings of small dimension. The proof of the following statement is essentially contained in \cite[2.6]{Page}.

\begin{prop}\label{FrobCsl}
Assume $R$ is normal and has a canonical ideal $\omega$. Then $\cx_F(R)\ls \sell(\omega^{(-1)})-1$.
\end{prop}

We now state our contribution towards Question \ref{EYquestion}.

\begin{cor}\label{FrobCor}
Assume $R$ is normal and  has a canonical ideal $\omega$ (e.g. $R$ is complete). 
\begin{enumerate}
\item[(1)] If $\dim R\ls 3$, then $\cx_F(R)\ls 1$.
\item[(2)] Moreover, if $\dim R = 4$, $R$ is analytically unramified and formally equidimensional, then  $\cx_F(R)\ls \ell(\omega^{(-1)})<\infty$. 
\item[(3)] If $R$ is analytically unramified, formally equidimensional, and Gorenstein locally  on $\Spec R\setminus \{\fm\}$, then $\cx_F(R)\ls \ell(\omega^{(-1)})<\infty$.
\end{enumerate}
\end{cor}
\begin{proof}
 The result  follows from Proposition \ref{FrobCsl} and, Corollary \ref{duttaCor} for (1), Corollary \ref{depthCor} for (2), and Corollary \ref{divisor} for (3).
\end{proof}

 \section{A bound for monomial ideals}\label{monoSec}

Let $R=k[x_1,\ldots, x_d]$ be a polynomial ring over the field $k$ and $\fm=(x_1,\ldots, x_d)$. 
%Let $\{\fe_1,\ldots, \fe_d\}$ be the standard bases of $\RR^d$ and 
 For $(a_1,\ldots, a_d)\in \NN^d$ let $\fx^\fa$ denote the monomial  $x_1^{a_1}\cdots x_d^{a_d}$. 
For a monomial prime ideal $\fp$ we denote by $\supp(\fp):=\{i\mid x_i\in \fp\}$ its {\it support}.

Let $\bght I$ denote the {\it big height} of  the ideal $I$, i.e., the largest height of an ideal in $\Min(I)$.
 
\begin{thm}\label{mainMonomial}
Let  $R=k[x_1,\ldots, x_d]$  and $I$ a monomial ideal, then 
$$\sell(I)\ls \ d - \left\lfloor\frac{d-1}{\bght I}\right\rfloor.$$
%Moreover, the limit $$\lim_{n\rightarrow \infty}\frac{\mu(I^{(n)})}{n^{\sell(I)}}$$ exists. 
\end{thm}

\begin{proof}
By Proposition \ref{ariRank} we may assume $\dim R/I\gs 1$.  Let $\Min(I)=\{\fp_1,\ldots, \fp_c\}$, we may further assume $\cup_{i=1}^c \supp(\fp_i)=[d]:=\{1,\ldots, d\}$. Let  $\{Q_1,\ldots, Q_c\}$ be the corresponding primary components of $I$ so that $I^{(n)}=\cap_{j=1}^c (Q_j^nR_{\fp_j}\cap R)$ for every $n\gs 1$. Notice that $\fx^\fa$ is a minimal generator of $I^{(n)}$ if and only if $\frac{\fx^\fa}{1}\in \cap_{j=1}^c Q_j^nR_{\fp_j}$  but $\frac{\fx^\fa}{x_i}$ is not in some $Q_j^nR_{\fp_j}$ for each $i\in [d]$. Equivalently, if there exists a function $\phi:[d]\rightarrow  [c]$ such that $\fx^\fa$ is in the set
$$S_\phi(n)=\{\fx^\fa\mid \frac{\fx^\fa}{1}\in \cap_{j=1}^c Q_j^nR_{\fp_j}\text{ and } \frac{\fx^\fa}{x_i}\not\in Q_{\phi(i)}^nR_{\fp_{\phi(i)}} \text{ for every } i\in [d]\}.$$  We  claim that for each of these functions $\phi$ we have  $|S_\phi(n)|=O(n^{d-|\Image \phi |})$; for this  we  may also assume $|S_\phi(n_0)|\neq 0$ for at least one $n_0\in \NN$. First we show that the claim  implies the result. If $i\in  \phi^{-1}(j)$ then there exists $\fa\in \NN^d$  and $n_0\in \NN$ with $\frac{\fx^\fa}{1}\in Q_j^{n_0}R_{\fp_j}$ and $\frac{\fx^\fa}{x_i}\not\in Q_j^{n_0}R_{\fp_j}$, it follows that  $i\in \supp(\fp_j)$. Therefore, $ |\supp( \fp_j)|\gs | \phi^{-1}(j)|$. Thus, $$|\Image \phi |(\bght I)\gs \sum_{j\in \Image \phi} \height \fp_j = \sum_{j\in \Image \phi} |\supp( \fp_j)|\gs \sum_{j\in\Image \phi}| \phi^{-1}(j)|=d.$$
 From this chain of  inequalities we obtain $d-|\Image \phi | +1\ls d-\lceil\frac{d}{\bght I}\rceil+1=d - \left\lfloor\frac{d-1}{\bght I}\right\rfloor$.

Now we prove the claim. For each $\Lambda\subseteq [d]$ let $\pi_\Lambda:R\rightarrow R$ be defined by $\pi_\Lambda(x_i)=1$ if $i\in \Lambda$ and $\pi_\Lambda(x_i)=x_i$ otherwise. Fix $\phi$ as above and set $F_j=[d]\setminus \phi^{-1}(j)$ for each $j\in\Image \phi$.  Then $\fx^\fa\in S_\phi(n)$ if and only if $\pi_{F_j}(\fx^\fa)$ is a minimal generator of $ \pi_{ F_j}(Q_j^n)= \pi_{ F_j}(Q_j)^n$ for every $j$. 
Notice that $\fx^\fa=\prod_{j\in \Image(\phi)}\pi_{F_j}(\fx^\fa)$, hence
$$|S_\phi(n)|\ls \prod_{j\in \Image \phi}\mu(\pi_{ F_j}(Q_j)^n)=\prod_{j\in \Image \phi}O\big(n^{\ell( \pi_{F_j}(Q_j))-1} \big)=O(n^{d-|\Image \phi|}),$$
as desired.
\end{proof}

As a corollary we recover the following result of Lyubeznik.

\begin{cor}[{\cite[Theorem]{Lyu86}}]\label{LyuCor}
Let  $R=k[x_1,\ldots, x_d]$  and $I$ a monomial ideal, then 
$$\ara(I)\ls \ d - \left\lfloor\frac{d-1}{\bght I}\right\rfloor.$$
\end{cor}
\begin{proof}
The result follows from Proposition \ref{ariRank} and Theorem \ref{mainMonomial}.
\end{proof}

We also apply Theorem \ref{mainMonomial} to the recent results in \cite[2.4]{HKTT} and \cite[3.3]{NT} to obtain the following interesting consequence.

\begin{cor}\label{depthBound}
Let  $R=k[x_1,\ldots, x_d]$  and $I$ a monomial ideal. Assume $I^{(n)}$ is integrally closed for $n\gg 0$.  Then for every  $n$ such that $I^{(n)}$ is integrally closed we have 
$$\depth R/I^{(n)}\gs \left\lfloor\frac{d-1}{\bght I}\right\rfloor\quad \text{and} \quad \projdim R/I^{(n)}\ls d-\left\lfloor\frac{d-1}{\bght I}\right\rfloor\quad.$$
In particular, the conclusion holds for every squarefree monomial ideal $I$ and every $n\in \ZZ_{>0}$.
\end{cor}

\begin{remark}
In \cite{Lyu84}, Lyubeznik shows examples of monomial ideals $I$ for each value of $d$ and $\bght I$ such that $\ara(I)= \ d - \left\lfloor\frac{d-1}{\bght I}\right\rfloor.$ These examples show that the  bounds in Theorem \ref{mainMonomial}, and Corollaries \ref{LyuCor} and \ref{depthBound} are sharp.
\end{remark}

\section{Some open questions}

In this section we briefly discuss some simple open questions that arise from our work. As before, we assume $(R,\fm,k)$ is a Noetherian local ring and $I$ and $R$-ideal.

We begin by  proposing the following conjecture. 

\begin{conjecture}\label{conj1}
$\sell(I)\ls \dim R$. 

\end{conjecture}

We do not know whether this holds true even over regular local rings.  In fact, it is not clear to us how to prove that there is some constant $C$, depending only on $R$, such that we always have $\sell(I)\ls C$; or even if $\sell(I)$ is always finite.

A bolder statement would be:
\begin{question}\label{theConjecture}
Is it true that $\sell(I)\ls \ell(I)$? 
\end{question}

We remark that even when $\R^s(I)$ is Noetherian we do not know whether Question \ref{theConjecture} has an affirmative answer, except when $\ell(I)\gs \dim R-1$ (see Proposition \ref{ariRank}). Further affirmative cases  are included in  Example \ref{determ}, Proposition \ref{extr}, and Corollary  \ref{duttaCor} (2). In addition, Corollaries \ref{depthCor}, \ref{ciCor}, and Corollary \ref{duttaCor} (1) show particular cases of the weaker inequality $\sell(I)\ls \ell(I)+1$.

The classical Burch-Brodmann inequality states that $\ell(I)+\displaystyle\lim_{n\to \infty}\depth R/I^n\ls \dim (R)$. We propose the following symbolic analogue.

\begin{question}\label{depthConj}
Is it true that $\sell(I)+ \displaystyle \liminf_{n\to \infty}\, \depth R/I^{(n)}\ls \dim(R)$?
\end{question}

First, we remark that if $\R^s(I)$ is Noetherian then the inequality in  \ref{depthConj} holds (see for instance \cite[9.23]{BV}). In fact, it suffices to assume that the algebra $\G^s(I):=\oplus_{n\in \NN}I^{(n)}/I^{(n+1)}$ is Noetherian, and in this case equality occurs if $\G^s(I)$ is Cohen-Macaulay. We note that some special cases  \ref{depthConj} follow from Corollary \ref{duttaCor}. In addition, the equality was shown to hold for squarefree monomial ideals in \cite[2.4]{HKTT}, and more generally when $I^{(n)}$ is integrally closed for every $n\gg 0$ in \cite[3.3]{NT}.

It would be very interesting to settle Questions \ref{theConjecture} and \ref{depthConj} even for prime ideals in a regular local ring. 

In Example \ref{notariBound} we show that $\ara(I)$ is not always bounded by $\sell(I)$, even for prime ideals. However, the ring therein is not regular. Thus, we ask the following.

\begin{question}
Assume $R$ is regular and $I$ is prime. Is $\ara(I)\ls \sell(I)$?
\end{question}

\section*{Acknowledgments}

We are grateful to  Florian Enescu, Mel Hochster, Kazuhiko Kurano, Linquan Ma, Paolo Mantero, Thomas Polstra, Kevin Tucker, and Bernd Ulrich for helpful discussions and encouragements.

\end{document}